\documentclass[10pt]{article}
\usepackage[utf8]{inputenc}

\usepackage{amssymb,amsthm,amsmath}
\usepackage{enumerate}
\usepackage{graphicx,color}
\usepackage{setspace}
\usepackage[hidelinks]{hyperref}

\newcommand{\dd}{\mathrm{d}}
\newcommand{\E}{\mathbb{E}}
\newcommand{\1}{\textbf{1}}
\newcommand{\R}{\mathbb{R}}

\newcommand{\red}{}

\DeclareMathOperator{\Var}{Var}

\usepackage[paper=a4paper, left=1.5in, right=1.5in, top=1.3in, bottom=1.4in]{geometry}
\newtheorem{theorem}{Theorem}
\newtheorem{lemma}[theorem]{Lemma}
\newtheorem{corollary}[theorem]{Corollary}

\theoremstyle{remark}
\newtheorem{remark}[theorem]{Remark}

\newtheorem{conjecture}{Conjecture}

\theoremstyle{definition}

\title{\vspace{-3em}Log-concavity and log-convexity of moments of averages of i.i.d. random variables}
\author{Philip Lamkin\thanks{Carnegie Mellon University; Pittsburgh, PA 15213, USA. Email: plamkin@andrew.cmu.edu}
\ and 
Tomasz Tkocz\thanks{\linespread{1.0} Carnegie Mellon University; Pittsburgh, PA 15213, USA. Email: ttkocz@math.cmu.edu. Research supported in part by the Collaboration Grants from the Simons Foundation.}
}
\date{21st April 2020}

\begin{document}

\maketitle

\begin{abstract}
We show that the sequence of moments of order less than 1 of averages of i.i.d. positive random variables is log-concave. For moments of order at least 1, we conjecture that the sequence is log-convex and show that this holds eventually for integer moments (after neglecting the first $p^2$ terms of the sequence).
\end{abstract}

\bigskip

\begin{footnotesize}
\noindent {\em 2010 Mathematics Subject Classification.} Primary 05A20; Secondary 60E15, 26D15.

\noindent {\em Key words.} log-convexity, log-concavity, moment comparison, sums of independent random variables
\end{footnotesize}

\bigskip

\section{Introduction and results}

Suppose $X_1, X_2, \ldots$ are i.i.d. copies of a positive random variable and $f$ is a nonnegative function. This article is concerned with certain combinatorial properties of the sequence
\begin{equation}\label{eq:def-an}
a_n = \E f\left(\frac{X_1+\dots+X_n}{n}\right), \qquad n = 1, 2, \dots.
\end{equation}
For instance, $f(x) = x^p$ is a fairly natural choice leading to the sequence of moments of averages of the $X_i$. Since we have the identity
\[
\sum_{i=1}^{n+1} x_i = \sum_{i=1}^{n+1}\frac{\sum_{j: j \neq i} x_j}{n},
\]
we conclude that the sequence $(a_n)_{n=1}^\infty$ is nonincreasing when $f$ is convex. What about inequalities involving more than two terms?

Such inequalities have been studied to some extent. One fairly general result is due to Boland, Proschan and Tong from \cite{BPT} (with applications in reliability theory). It asserts in particular that for $n = 2, 3, \ldots$,
\begin{equation}\label{eq:Boland}
\E \phi(X_1+\dots+X_n,X_{n+1}+\dots+X_{2n}) \leq \E \phi(X_1+\dots+X_{n-1},X_{n}+\dots+X_{2n})
\end{equation}
for a symmetric (invariant under permuting coordinates) continuous random vector $X = (X_1,\dots,X_{2n})$ with nonnegative components and a symmetric convex function $\phi: [0,+\infty)^2 \to \R$.

We obtain a satisfactory answer to a natural question of log-convexity/concavity of sequences $(a_n)$ for completely monotone functions, also providing insights into the case of power functions.


Recall that a nonnegative sequence $(x_n)_{n=1}^\infty$ supported on a set of contiguous integers is called log-convex (resp. log-concave) if $x_n^2 \leq x_{n-1}x_{n+1}$ (resp. $x_n^2 \geq x_{n-1}x_{n+1}$) for all $n \geq 2$ (for background on log-convex/concave sequences, see for instance \cite{LW, Stanley}). One of the crucial properties of log-convex sequences is that log-convexity is preserved by taking sums (which follows from the Cauchy-Schwarz inequality, see for instance \cite{LW}). Recall that an infinitely differentiable function function $f: (0,\infty) \to (0,\infty)$ is called completely monotone if we have $(-1)^nf^{(n)}(x) \geq 0$ for all positive $x$ and $n = 1, 2, \ldots$; equivalently, by Bernstein's theorem (see for {\red instance} \cite{Fel}), the function $f$ is the Laplace transform of a nonnegative Borel measure $\mu$ on $[0,+\infty)$, that is
\begin{equation}\label{eq:Lap}
f(x) = \int_0^\infty e^{-tx}\dd\mu(t).
\end{equation}
For example, when $p<0$, the function $f(x) = x^p$ is completely monotone. Such integral representations are at the heart of our first two results.

\begin{theorem}\label{thm:comp-mon}
Let $f\colon (0,\infty)\rightarrow (0,\infty)$ be a completely monotone function. Let $X_1, X_2, \ldots$ be i.i.d. positive random variables. Then the sequence $(a_n)_{n=1}^\infty$ defined by \eqref{eq:def-an} is log-convex.
\end{theorem}

\begin{theorem}\label{thm:comp-mon'}
Let $f\colon [0,\infty)\rightarrow [0,\infty)$ be such that $f(0) = 0$ and its derivative $f'$ is completely monotone. Let $X_1, X_2, \ldots$ be i.i.d. nonnegative random variables. Then the sequence $(a_n)_{n=1}^\infty$ defined by \eqref{eq:def-an} is log-concave.
\end{theorem}

In particular, applying these to the functions $f(x) = x^p$ with $p < 0$ and $0 < p < 1$ respectively, we obtain the following corollary.

\begin{corollary}\label{cor:p<1}
Let $X_1, X_2, \ldots$ be i.i.d. positive random variables. The sequence
\begin{equation}\label{eq:seq^p}
b_n = \E\left(\frac{X_1+\dots+X_n}{n}\right)^p, \qquad n=1,2,\ldots
\end{equation}
is log-convex when $p < 0$ and log-concave when $0 < p < 1$.
\end{corollary}

For $p > 1$, we pose the following conjecture.

\begin{conjecture}\label{conj:p>1}
Let $p > 1$. Let $X_1, X_2, \ldots$ be i.i.d. nonnegative random variables. Then the sequence $(b_n)$ defined in \eqref{eq:seq^p} is log-convex.
\end{conjecture}

We offer a partial result supporting this conjecture.

\begin{theorem}\label{thm:p>1}
Let $X_1, X_2, \ldots$ be i.i.d. nonnegative random variables, let $p$ be a positive integer and let $b_n$ be defined by \eqref{eq:seq^p}. Then for every $n \geq p^2$, we have
$b_n^2 \leq b_{n-1}b_{n+1}$.
\end{theorem}

\begin{remark}\label{rem:p23}
When $p = 2$, we have $b_n = \frac{n\E X_1^2  + n(n-1)(\E X_1)^2}{n^2} = (\E X_1)^2 + n^{-1}\Var(X_1)$, which is clearly a log-convex sequence (as a sum of two log-convex sequences). 
The following argument for $p=3$ was kindly communicated to us by Krzysztof Oleszkiewicz: 
when $p=3$, we can write
\[
b_n = (\E X_1)^3 + \big(\E X_1^3 + (\E X_1)^3 - 2(\E X_1^2)(\E X_1)\big)n^{-2} + (\E X_1)\Var(X_1)(3n^{-1}-n^{-2}).
\]
The sequences $(n^{-2})$ and $(3n^{-1}-n^{-2})$ are log-convex. By the Cauchy-Schwarz inequality the factor at $n^{-2}$ is nonnegative,
\[
\E X_1^2 \leq \sqrt{\E X_1^3}\sqrt{\E X_1} \leq \frac{\E X_1^3}{2\E X_1} + \frac{(\E X_1)^2}{2},
\]
so again $(b_n)$ is log-convex as a sum of three log-convex sequences. 
It remains elusive how to group terms and proceed along these lines in general. Our proof of Theorem \ref{thm:p>1} relies on this idea, but uses a  straightforward way of rearranging terms.
\end{remark}

\begin{remark}\label{rem:Boland}
It would be tempting to use the aforementioned result of Boland et al. with $\phi(x,y) = (xy)^p$ to resolve Conjecture \ref{conj:p>1}. However, this function is neither convex nor concave on $(0,+\infty)^2$ for $p > \frac{1}{2}$. For $0 < p < \frac{1}{2}$, the function is concave and \eqref{eq:Boland} yields $(b_nn^p)^2 \geq b_{n-1}(n-1)^pb_{n+1}(n+1)^p$, $n \geq 2$, equivalently, $b_n^2 \geq \left(\frac{n^2-1}{n^2}\right)^pb_{n-1}b_{n+1}$. Corollary~\ref{cor:p<1} improves on this by removing the factor $\left(\frac{n^2-1}{n^2}\right)^p < 1$. For $p < 0$, $\phi$ is convex, so \eqref{eq:Boland} gives $b_n^2 \leq \left(\frac{n^2-1}{n^2}\right)^pb_{n-1}b_{n+1}$ and Corollary \ref{cor:p<1} removes the factor $\left(\frac{n^2-1}{n^2}\right)^p > 1$.
\end{remark}

Concluding this introduction, it is of significant interest to study the log-behaviour of various sequences, particularly those emerging from algebraic, combinatorial, or geometric structures, which has involved and prompted the development of many deep and interesting methods, often useful beyond the original problems (see, e.g., 
\cite{Brenti, D, DD, Fad, NO, Seg, Stanley, Stanley2}
). We propose to consider sequences of moments of averages of i.i.d. random variables arising naturally in probabilistic limit theorems. For moments of order less than $1$, we employ an analytical approach exploiting integral representations for power functions. For moments of order higher than $1$, our Conjecture \ref{conj:p>1}, besides refining the monotonicity property of the sequence $(b_n)$ (resulting from convexity), would furnish new examples of log-convex sequences. For instance, neither does it seem trivial, nor handled by known techniques, to determine whether the sequence obtained by taking the Bernoulli distribution with parameter $\theta \in (0,1)$, $b_n = \sum_{k=0}^n \binom{n}{k}\left(\frac{k}{n}\right)^p\theta^k(1-\theta)^{n-k}$ is log-convex. In the case of integral $p$, we have $b_n = \sum_{k=0}^p S(p,k)\frac{n!}{(n-k)!n^p}\theta^k$, where $S(p,k)$ is the Stirling number of the second kind. 

The rest of this paper is occupied with the proofs of Theorems \ref{thm:comp-mon}, \ref{thm:comp-mon'}, \ref{thm:p>1} (in their order of statement) and then we conclude with additional remarks and conjectures.

\section{Proofs}

\subsection{Proof of Theorem \ref{thm:comp-mon}}

Suppose that $f$ is completely monotone.
Using \eqref{eq:Lap} and independence, we have
\[
a_n = \E f\left(\frac{X_1+\ldots+X_n}{n}\right) = \int_0^\infty \left[\E e^{-tX_1/n}\right]^n \dd\mu(t).\]
Let $u_n(t) = \left[\E e^{-tX_1/n}\right]^n$. It suffices to show that for every positive $t$, the sequence $(u_n(t))$ is log-convex (because sums/integrals of log-convex sequences are log-convex: the Cauchy-Schwarz inequality applied to the measure $\mu$ yields
\[
\left(\int \sqrt{u_{n-1}(t)u_{n+1}(t)} \dd\mu(t)\right)^2 \leq \left(\int u_{n-1}(t) \dd\mu(t)\right)\left(\int u_{n+1}(t) \dd\mu(t)\right),
\]
which combined with $u_n(t) \leq \sqrt{u_{n-1}(t)u_{n+1}(t)}$, gives $a_n^2 \leq a_{n-1}a_{n+1}$).
The log-convexity of $(u_n(t))$ follows from H\"older's inequality,
\[
\E e^{-tX_1/n} = \E e^{-\frac{n-1}{2n}\frac{tX_1}{n-1}}e^{-\frac{n+1}{2n}\frac{tX_1}{n+1}} \leq \left(\E e^{-\frac{tX_1}{n-1}}\right)^{\frac{n-1}{2n}}\left(\E e^{-\frac{tX_1}{n+1}}\right)^{\frac{n+1}{2n}},\]
which finishes the proof.

\subsection{Proof of Theorem \ref{thm:comp-mon'}}

Suppose now that $f(0) = 0$ and $f'$ is completely monotone, say $f'(x) = \int_0^\infty e^{-tx}\dd\mu(t)$ for some nonnegative Borel measure $\mu$ on $(0,\infty)$ (by \eqref{eq:Lap}). Introducing a new measure $\dd\nu(t) = \frac{1}{t}\dd\mu(t)$ we can write
\[
f(y) = f(y) - f(0) = \int_0^y f'(x) \dd x = \int_0^\infty\int_0^\infty te^{-tx}\1_{\{0<x<y\}}\dd x\dd\nu(t).\]
Integrating against $\dd x$ gives
\[
f(y) = \int_0^\infty \big[1-e^{-ty}\big]\dd\nu(t).\]
Let $F$ be the Laplace transform of $X_1$, that is
\[
F(t) = \E e^{-tX_1}, \qquad t > 0.
\]
Then
\[
\E f\left(\frac{X_1+\ldots+X_n}{n}\right) = \int_0^\infty \Big[ 1-F(t/n)^n \Big]\dd\nu(t) = \int_0^\infty G(n,t) \dd\nu(t),
\]
where, to shorten the notation, we introduce the following nonnegative function
\[
G(\alpha,t) = 1 - F(t/\alpha)^\alpha, \qquad \alpha, t > 0.
\]
To show the inequality
\[ 
\Big[\E f\left(\frac{X_1+\ldots+X_n}{n}\right) \Big]^2 \geq \E f\left(\frac{X_1+\ldots+X_{n-1}}{n-1}\right)\cdot\E f\left(\frac{X_1+\ldots+X_{n+1}}{n+1}\right)
\]
it suffices to show that pointwise
\[
G(n,s)G(n,t) \geq \frac{1}{2}G(n-1,s)G(n+1,t) + \frac{1}{2}G(n+1,s)G(n-1,t),\]
for all $s, t > 0$. This follows from two properties of the function $G$:

\noindent
1) for every fixed $t>0$ the function $\alpha \mapsto G(\alpha,t)$ is nondecreasing,

\noindent
2) the function $G(\alpha,t)$ is concave on $(0,\infty)\times (0,\infty)$.

Indeed, by 2) we have
\[
G(n,s)G(n,t)\geq \frac{G(n-1,s)+G(n+1,s)}{2}\cdot\frac{G(n-1,t)+G(n+1,t)}{2}\]
{\red (in fact we only use concavity in the first argument).} It thus suffices to prove that
\begin{align*}
& G(n-1,s)G(n-1,t)+G(n+1,s)G(n+1,t) \\
&\quad- G(n-1,s)G(n+1,t) - G(n+1,s)G(n-1,t) \\
&= \Big[G(n-1,s)-G(n+1,s)\Big]\cdot \Big[G(n-1,t)-G(n+1,t)\Big]
\end{align*}
is nonnegative, which follows by 1).

It remains to prove 1) and 2). To prove the former, first notice that $F(t/\alpha)^\alpha = \left(\E e^{-tX/\alpha}\right)^{\alpha} = {\red \|e^{-tX}\|_{1/\alpha}}$ {\red is the $L_{1/\alpha}$-norm of $e^{-tX}$. By convexity, for an arbitrary random variable $Z$, $p \mapsto (\E|Z|^p)^{1/p} = \|Z\|_p$ is nondecreasing, so $F(t/\alpha)^\alpha = \|e^{-tX}\|_{1/\alpha}$ is nonincreasing and thus $G(\alpha,t) = 1 - F(t/\alpha)^\alpha$ is nondecreasing.} To prove the latter, notice that by H\"older's inequality the function $t \mapsto \log F(t)$ is convex, 
\[
{\red F(\lambda s + (1-\lambda) t) = \E [(e^{-sX})^\lambda (e^{-tX})^{1-\lambda}] \leq (\E e^{-sX})^\lambda(\E e^{-tX})^{1-\lambda} = F(s)^\lambda F(t)^{1-\lambda}}.
\] 
Therefore its perspective function $H(\alpha,t) = \alpha \log F(t/\alpha)$ is convex (see, e.g. Ch. 3.2.6 in \cite{BV}), which implies that $F(t/\alpha)^{\alpha} = e^{H(\alpha,t)}$ is also convex.

\subsection{Proof of Theorem \ref{thm:p>1}}

We recall a standard combinatorial formula: first by the multinomial theorem and independence, we have
\[
\E (X_1+\dots+X_n)^p = \sum \frac{p!}{p_1!\cdots p_n!}\E (X_1^{p_1}\cdot\ldots\cdot X_{n}^{p_n}) = \sum \frac{p!}{p_1!\cdots p_n!} \mu_{p_1}\cdot\ldots\cdot\mu_{p_n},
\]
where the sum is over all sequences $(p_1,\dots,p_n)$ of nonnegative integers such that $p_1+\dots+p_n = p$ and we denote $\mu_k = \E X_1^k$, $k \geq 0$. Now we partition the summation according to the number $m$ of positive terms in the sequence $(p_1,\ldots,p_n)$. {\red Let $Q_m$ be the set of integer partitions of $p$ into exactly $m$ (nonempty) parts, that is the set of $m$-element multisets $q = \{q_1,\dots, q_m\}$ with positive integers $q_j$ summing to $p$, $q_1+\dots + q_m = p$. Then} 
\[
\E (X_1+\dots+X_n)^p = \sum_{m=1}^p\sum_{q \in Q_m}\frac{p!}{q_1!\cdots q_m!}\frac{n!}{\alpha(q)\cdot(n-m)!}\mu_{q_1}\cdots\mu_{q_m},
\]
where $\alpha(q) = l_1!\cdots l_h!$ for $q = \{q_1,\ldots,q_m\}$ with exactly $h$ distinct terms such that there are $l_1$ terms of type $1$, $l_2$ terms of type $2$, etc., so $l_1+\dots+l_h = m$ (e.g. for $q = \{2,2,2,3,4,4\} \in Q_6$, we have $h=3$, $l_1 = 3$, $l_2 = 1$, $l_3=2$). The factor $\frac{n!}{\alpha(q)\cdot(n-m)!}$ arises because given a {\red multiset} $q \in Q_m$, there are exactly
\begin{align*}
\binom{n}{l_1}\binom{n-l_1}{l_2}&\binom{n-l_1-l_2}{l_3}\dots\binom{n-l_1-\dots-l_{h-1}}{l_h} \\
&= \frac{n!}{l_1!\cdot\ldots\cdot l_h!\cdot (n-l_1-\ldots-l_h)!} = \frac{n!}{\alpha(q)\cdot (n-m)!}
\end{align*}
many nonnegative integer-valued sequences $(p_1,\ldots,p_n)$ with $\mu_{p_1}\cdots\mu_{p_n} = \mu_{q_1}\cdots\mu_{q_m}$ (equivalently, $\{p_1,\ldots,p_n\} = \{q_1,\ldots,q_m,0\}$, as sets). 

We have obtained
\begin{equation}\label{eq:ESn}
b_n = \E \left(\frac{X_1+\dots+X_n}{n}\right)^p = \sum_{m=1}^p \frac{n!}{n^p(n-m)!} \sum_{q \in Q_m}\beta(q){\red \mu(q)},
\end{equation}
where $\beta(q) = \frac{p!}{\alpha(q)\cdot q_1!\cdots q_m!}$ and $\mu(q) = \mu_{q_1}\cdots\mu_{q_m}$. By homogeneity, we can assume that $\mu_1 = \E X_1 = 1$. Note that when $X_1$ is constant, we get from \eqref{eq:ESn} that
\[
1 = \sum_{m=1}^p \frac{n!}{n^p(n-m)!} \sum_{q \in Q_m}\beta(q).
\]
Since $Q_p$ has only one element, namely $\{1,\ldots,1\}$ and $\mu(\{1,\ldots,1\}) = 1$, when we subtract the two equations, the terms corresponding to $m=p$ cancel and we get
\[
b_n - 1 = \sum_{m=1}^{p-1} \frac{n!}{n^p(n-m)!}\sum_{q \in Q_m} \beta(q)(\mu(q)-1).
\]
By the monotonicity of moments, $\mu(q) \geq 1$ for every $q$, so $(b_n)$ is a sum of the constant sequence $(1,1,\ldots)$ and the sequences $(u_n^{(m)}) = (\frac{n!}{n^p(n-m)!})$, $m=1,\ldots,p-1$, multiplied respectively by the nonnegative factors $\sum_{q \in Q_m} \beta(q)(\mu(q)-1)$. Since sums of log-convex sequences are log-convex, it remains to verify that for each $1 \leq m \leq p-1$, we have $(u_n^{(m)})^2 \leq u_{n-1}^{(m)}u_{n+1}^{(m)}$, $n \geq p^2$.  The following lemma finishes the proof.

\begin{lemma}\label{lm:seq}
Let $p \geq 2$, $1 \leq m \leq p-1$ be integers. Then the function
\[
f(x) = \log\frac{x(x-1)\cdots(x-m+1)}{x^p}
\]
is convex on $[p^2-1,\infty)$.
\end{lemma}
\begin{proof}
The statement is clear for $m=1$. Let $2 \leq m \leq p-1$ and $p \geq 3$. We have
\[
x^2f''(x) = p-1 - x^2\sum_{k=1}^{m-1} \frac{1}{(x-k)^2}.
\]
To see that this is positive for every $x \geq p^2-1$ and $2 \leq m \leq p-1$, it suffices to consider $m = p-1$ and $x = p^2-1$ (writing $\frac{x}{x-k} = 1 + \frac{k}{x-k}$, we see that the right hand side is increasing in $x$). Since
\begin{align*}
\sum_{k=1}^{p-2} \frac{1}{(p^2-1-k)^2} = \sum_{k=p^2-p+1}^{p^2-2} \frac{1}{k^2} \leq \sum_{k=p^2-p+1}^{p^2-2} \left(\frac{1}{k-1} - \frac{1}{k}\right) &= \frac{1}{p^2-p} - \frac{1}{p^2-2},
\end{align*}
we have
\[
x^2f''(x) \geq p-1-(p^2-1)^2\left(\frac{1}{p^2-p} - \frac{1}{p^2-2}\right) = \frac{(p-1)(p+2)}{p(p^2-2)},
\]
which is clearly positive.
\end{proof}

\section{Final remarks}

\begin{remark}
Using majorization type arguments (see, e.g. \cite{MO}), Conjecture \ref{conj:p>1} can be verified in a rather standard but lengthy way for every $p>1$ and $n=2$. The idea is to establish a pointwise inequality: we conjecture that for nonnegative numbers $x_1,\ldots,x_{2n}$ and a convex function  $\phi: [0,\infty)\to[0,\infty)$ we have
\[ 
\frac{1}{\binom{2n}{n}}\sum_{|I|=n}\phi\left(\frac{x_Ix_{I^c}}{n^2}\right) \leq \frac{1}{\binom{2n}{n+1}}\sum_{|I|=n+1}\phi\left(\frac{x_Ix_{I^c}}{n^2-1}\right),\]
where for a subset $I$ of the set $\{1,\ldots,2n\}$ we denote $x_I = \sum_{i\in I} x_i$. We checked that this holds for $n=2$. Taking the expectation on both sides for $\phi(x) = x^p$ gives the desired result that $b_n^2 \leq b_{n-1}b_{n+1}$.
\end{remark}

\begin{remark}
It is tempting to ask for generalisations of Conjecture \ref{conj:p>1} beyond the power functions, say to ask whether the sequence $(a_n)$ defined in \eqref{eq:def-an} is log-convex for every convex function $f$. This is false, as can be seen by taking the function $f$ of the form $f(x) = \max\{x-a,0\}$ and the $X_i$ to be i.i.d Bernoulli random variables. 
\end{remark}

\paragraph{Acknowledgments.} We thank Krzysztof Oleszkiewicz for his great help, valuable correspondence, feedback and for allowing us to use his slick proof for $p=3$ from Remark 5. We also thank Marta Strzelecka and and Micha\l\ Strzelecki for the helpful discussions. {\red We are indebted to the referees for their careful reading and useful comments.}

We are grateful for the excellent working conditions provided by the Summer Undergraduate Research Fellowship 2019 at Carnegie Mellon University during which this work was initiated.

\end{document}